\documentclass[reqno]{amsart}
\usepackage{amscd}
\usepackage[dvips]{graphics}

\def\beq{\begin{equation}}
\def\eeq{\end{equation}}
\def\ba{\begin{array}}
\def\ea{\end{array}}

\newtheorem{thm}{Theorem}[section]
\newtheorem{lm}[thm]{Lemma}
\newtheorem{prop}[thm]{Proposition}
\newtheorem{crl}[thm]{Corollary}

\theoremstyle{definition}

\theoremstyle{remark}
 
\numberwithin{equation}{section}

\begin{document}
\pagestyle{plain}  
\title{Prescribing integral curvature equation}


\author  {Meijun Zhu}
\address{Department of Mathematics, The University of Oklahoma, Norman, OK 73019, USA}


\begin{abstract}
In this paper we formulate  new curvature functions on $\mathbb{S}^n$ via integral operators. For certain even orders, these curvature functions are equivalent to the classic curvature functions defined via differential operators, but not for all even orders. Existence result for antipodally symmetric prescribed curvature functions on $\mathbb{S}^n$ is obtained. As a corollary, the existence of a conformal metric for an antipodally symmetric prescribed $Q-$curvature functions on $\mathbb{S}^3$ is proved. Curvature function on general compact manifold as well as the conformal covariance property for the corresponding integral operator are also addressed.

\end{abstract}

 \maketitle

\section{Formulation of the problem and the main results}
The main problem we will consider in this paper is the solvability of the following integral equation:
\begin{equation}
u(\xi)=\int_{\mathbb{S}^n}R(\eta)u^{\frac{n+\alpha}{n-\alpha}}(\eta)|\xi-\eta|^{\alpha-n}d\eta, \ \ u>0.\ \ \label{new1-1}
\end{equation}
where $\alpha \ne n$ is a given positive parameter, $R(\xi)$ is a given positive function,  and throughout the paper $|\xi-\eta|$ is denoted as the chordal distance from $\xi$ to $\eta$ in $\mathbb{R}^{n+1}$. In this paper, we mainly consider the case of $\alpha>n$.

For $n=1$ and $\alpha=2$, equation \eqref{new1-1} is equivalent to
\begin{equation}
u_{\theta \theta}+\frac 14 u=\frac{R(\theta)}{u^{3}}, \ \ u>0 \ \ \mbox{on} \ \ \mathbb{S}^1.\label{new1-1-1}
\end{equation}
The above equation was studied in Ni and Zhu \cite{NZ1} as one dimensional conformal curvature equation, following from the analytic extension of conformal Laplacian operators on $\mathbb{S}^n$. Further, if we let $v(\theta)=u(2 \theta)$, the above equation becomes the affine curvature equation for a symmetric convex two dimensional  body:
\begin{equation}
v_{\theta \theta}+v=\frac{\tilde R(\theta)}{ v^{3}},  \ \ v>0 \ \ \mbox{on} \ \ \mathbb{S}^1. \label{new1-2}
\end{equation}
In general, if $v(\theta)$ is a solution to equation \eqref{new1-2} and satisfies the following orthogonal condition
\begin{equation}
\int_0^{2 \pi} \frac{\tilde R(\theta)\cos \theta}{ v^{3}}=\int_0^{2 \pi} \frac{\tilde R(\theta) \sin \theta}{ v^{3}}=0, \label{ortho}
\end{equation}
then $v(\theta)$ is the supporting function of a two dimensional  convex body and $\tilde R(\theta)$ is the affine curvature of the boundary curve. When $\tilde R(\theta)$ is not a constant, any solution to equation \eqref{new1-2} automatically satisfies orthogonal condition \eqref{ortho}, but may not be $\pi-$periodic function. See the related work by Ai, Chou and Wei \cite{ACW2000},
where they discussed the existence of $\pi-$period solution to equation \eqref{new1-2} for $\pi-$periodic positive function $\tilde R(\theta)$; and a consequent paper by Jiang, Wang and Wei \cite{JWW2011}.  For sign-changed $\tilde R(\theta)$, the equation was studied in \cite{Chen2006}, \cite{Jiang2010} and \cite{DZ1}, in which the existence for $2\pi/k$ (for integer $k>1$) periodic function was obtained. The problem is also related to the $L_2$ Minkowski problem which was initially studied  by Lutwak \cite{Lu1}.


For $n=3$ and $\alpha=4$, equation \eqref{new1-1} is equivalent to so-called prescribing $Q-$curvature problem on $\mathbb{S}^3$:
\begin{equation}
(-\Delta)^2 u+\frac 12 \Delta u-\frac {15}{16} u=\frac{R(\xi)}{u^{7}}, \ \ u>0 \ \ \mbox{on} \ \ \mathbb{S}^3.\label{new1-1-2}
\end{equation}
Though the corresponding sharp Sobolev inequality for Paneitz operator $P_2 (u):=(-\Delta)^2 u+\frac 12 \Delta u-\frac {15}{16} u$ was obtained ten more years ago by Yang and Zhu \cite{YZ2004} (see, also Hang and Yang \cite{HY2004}), not too much progress in the study of  prescribing $Q-$curvature problem on $\mathbb{S}^3$ was made, except that it is known that similar Kazdan-Warner obstruction does exist, see, e.g. Xu \cite{Xu2007}. In particular, the existence of positive solution for an antipodally symmetric curvature function $R(\xi)$ is unknown, even though it is conjectured for a while that there shall be some similar existence results to those of Moser \cite{Moser1973} and Escorbar and Schoen \cite{ES1986}.

Analytically, one may even consider the extension of such study for general differential operators (for example: polynomial Laplacian operators) or pseudo-differential operators (for example: fractional Laplacian operators); and for high dimensional cases (for example: prescribing $Q-$curvature problem on $\mathbb{S}^3$). This motivation is also consistent with some recent studies on finding suitable conformal operators on general manifolds with the leading differential operators given by $(-\Delta)^{\frac {\alpha}2}$, see, e. g. \cite{GJMS1992}, \cite{GZ2003},  \cite{CG2011}, and references therein.

Curvature equations involving high order derivatives (including $Q-$ curvature equations) and  fully nonlinear curvature  equations (such as $\sigma_k$ operators of Schouten tensor) have been extensively studied in the past decade, see, e.g. \cite{CGY2002}, \cite{GV2003}, \cite{Br2003}, \cite{LL2005},  \cite{GV2007}, \cite{DM2008} and references therein.  All these differential operators, such as Paneitz operators with even powers and $\sigma_k$ operators of Schouten tensor, are defined pointwise.

Recently, there have been many interesting results concerning the fractional Yamabe problem, as well as the fractional
prescribing curvature problem, see, e.g. \cite{GZ2003}, \cite{GM2011}, \cite{GQ2011}, \cite{JLX2011a}-\cite{JX2011}
and references therein. In these studies the notion for the globally defined fractional Paneitz operator $P_\alpha$
is used  and has a direct link to singular integral operators (see Caffarelli and Silvestre \cite{CS2007} for a new view point of fractional Laplacian operator). We may view the differential operator $(-\Delta)^k$ on $\mathbb{R}^n$ as the even power ($\alpha=2k$) Paneitz operator, and hope it to produce a nice curvature function (e.g. for $k=1$, one can introduce the  scalar curvature for a conformal metric; for $k=2$,  one may introduce the so-called $Q$-curvature, see, e.g. \cite{Bra2004} ,  \cite{BG2008} and \cite{NZ1}.) However, the differential operator $(-\Delta)^k$sometimes has zero eigenvalue (for example,  $(-\Delta)^2$ has a zero eigenvalue on $\mathbb{R}^2$, see, e.g. Hang \cite{H2007} or the example below), which makes it impossible to introduce a reasonable curvature function via such a differential operator: Using such a differential operator, the $Q$-curvature of standard sphere $\mathbb{S}^2$ is zero. See more details in the following example.

\smallskip

\noindent{\bf 1.1. An example: the advantage for introducing new curvature functions}

\smallskip

Consider the standard bubbling function for operator $(-\Delta)^2$ on $\mathbb{R}^2$: $u_\epsilon=(\epsilon^2+|x|^2)/\epsilon$. One can easily check that it  is not a solution to
$$(-\Delta)^2 u =u^{-3}, \quad \, u>0 \quad \, \, \mbox{in} \quad \, \mathbb{R}^2.$$
Note that $g_\epsilon= u_\epsilon^{-2}d x^2$ is a complete metric for $\mathbb{S}^2$. If we introduced the $Q\text{-curvature}$ for $(\mathbb{S}^2, g_\epsilon)$ as before by defining $Q(x)= (-\Delta)^2 u_\epsilon /u_\epsilon^{-2}$, then $(\mathbb{S}^2, g_\epsilon)$ would have {\it zero} $Q-$curvature.

Then, is there any curvature equation that the bubbling function may associate with?


It turns out that, up to a constant multiplier, this bubbling function $u_\epsilon$ satisfies the following integral equation
\begin{equation}\label{3-5}
u(x)=\int_{\mathbb{R}^{2}}{|x-y|^{2}} {u^{-3}(y)}dy, \quad \, u>0 \quad  \, \mbox{in} \, \,\mathbb{R}^{2}.
\end{equation}

This observation leads us to seek for the integral form for equation (\ref{new1-1-1}).

\medskip

\noindent{\bf 1.2. Equivalent integral equation for equation (\ref{new1-1-1}) on $\mathbb{S}^1$}

\smallskip

The standard approach to solve equation (\ref{new1-1-1}) is to find critical points for quotient functional:
$$
E_1(u)=\frac{\int_{\mathbb{S}^1}(u_\theta^2-\frac 14 u^{2})dS}{\int_{\mathbb{S}^1}Ru^{-2}d S},
$$
Under the assumption that $R(\theta)$ is antipodally symmetric (i.e. $R(\theta)=R(\theta+\pi)$ on $\mathbb{S}^1),$ one can rule out the possibility that a minimizing sequence may vanish at certain point on $\mathbb{S}^1$, thus  obtain the existence of a positive solution,  see, for example, W. Chen \cite{Chen2006}.
 This approach seems to work only for second order equation on one dimensional circle.

 Here, we shall initiate a completely new approach. Observe that equation (\ref{new1-1-1}) is equivalent to the following integral equation:
$$
f^{-1/3}(\theta)=\int_{\mathbb{S}^1}R(\gamma)f(\gamma)|2\sin\frac{\theta-\gamma}2|^{\alpha-1}d\gamma, \ \ f>0.\
$$
For positive function $R(\xi)$, this equation may be solved via finding the critical point to a different quotient functional
$$
J_2(f)=\frac{\int_{\mathbb{S}^1}\int_{\mathbb{S}^1}(R(\xi)R(\eta)f(\xi)f(\eta)|2\sin\frac{\xi-\eta}2| d\xi d\eta }{(\int_{0}^{2 \pi}Rf^{2/3}d \xi)^{3/2}}.
$$

\medskip

 For general $\alpha>0$, we can introduce a new curvature function on a compact manifold conformally equivalent to the standard sphere via an integral form, which covers the above mentioned $Q-$curvtuare function as a special case.

 \smallskip

\noindent{\bf 1.3. Curvature functions and curvature equations}

\smallskip

%

 Let $(\mathbb{S}^n, g_0)$ be the standard sphere with the induced metric $g_0$ from $ \mathbb{R}^{n+1}$ and $\alpha \ne n$ be a positive parameter. For a given positive function $u\in C(\mathbb{S}^n)$, we define the new curvature function $Q_\alpha(\xi)\in L^1(\mathbb{S}^n)$ for metric $g=u^{\frac 4{n-\alpha}} g_0$  as a function implicitly given by
\begin{equation}\label{3-6}
u(\xi)= c_{n,\alpha}\int_{\mathbb{S}^{n}}{|\xi-\eta|^{\alpha-n}} {Q_\alpha (\eta)u^{\frac{n+\alpha}{n-\alpha}}(\eta)}dV_{g_0}:=\tilde I_\alpha ({Q_\alpha u^{\frac{n+\alpha}{n-\alpha}}}),
\end{equation}
where $|\xi-\eta|$ is the distance from point $\xi$ to point $\eta$ in $\mathbb{R}^{n+1}$,  $c_{n, \alpha}^{-1}=\int_{\mathbb{S}^n}|\xi-\eta|^{\alpha-n}dV_{g_0} =2^{\alpha-1}|\mathbb{S}^{n-1}|\frac{\Gamma(n/2)\Gamma(\alpha/2)}{\Gamma((n+\alpha)/2)}.$


This curvature function can also be defined on $\mathbb{R}^n$ via a stereographic projection. Let $\mathcal{S}:\, x \in \mathbb{R}^n\to \xi \in \mathbb{S}^n\backslash(0,0,\cdots,-1)$ be the inverse of the stereographic projection, defined by
\begin{eqnarray*}
\xi_j:=\frac{2\xi_j}{1+|\xi|^2},\quad\text{for}~j=1,2,\cdots,n;\quad \xi_{n+1}:=\frac{1-|x|^2}{1+|x|^2}.
\end{eqnarray*}
Using this stereographic projection, we can redefine the curvature function  on $\mathbb{R}^n$ as follows.  Let $u(\xi)=\big(\frac2{1+|x|^2}\big)^{\frac {n+\alpha}2}v(x),$  then
\begin{equation}\label{3-6-1}
v(x)= c_{n,\alpha}\int_{\mathbb{R}^{n}}{|x-y|^{\alpha-n}} {Q_\alpha (y)v^{\frac{n+\alpha}{n-\alpha}}(y)}dy:= I_\alpha ({Q_\alpha v^{\frac{n+\alpha}{n-\alpha}}}).
\end{equation}
Using Fourier transform, we can check  that  for $\alpha \in (0, n)$, equation \eqref{3-6-1} is generically equivalent to the following differential equation
\begin{equation}\label{3-6-2}
(-\Delta)^{\alpha /2} v={Q}_\alpha(\xi) v^{\frac{n+\alpha}{n-\alpha}}, \quad \, v>0 \quad  \ \mbox{on} \, \,\mathbb{R}^{n},
\end{equation}
for $v(x)$  with suitable decay rate at infinity, see, for example, Chen, Li and Ou \cite{CLO2005}.
 Thus, the new curvature $Q_\alpha(\xi)$ is uniquely determined by the metric $u^{4/(n-\alpha)}g_0$ for $\alpha\in(0, n)$.

 For $\alpha>n$, as we pointed out  by the above two dimensional example,  equation \eqref{3-6-1} may not be equivalent to equation \eqref{3-6-2}.  However, we can still introduce the same notion for curvature function $Q_\alpha(\xi)$ via equation \eqref{3-6} or \eqref{3-6-1}. In fact, the above curvature function is well defined for $\alpha-n \ne 2k$ since the inverse operator for $\tilde I_\alpha$ is well defined, see, for example, Pavlov and Samko
  \cite{PavlovS1984}), Rubin \cite{Rubin1992} and Samko \cite{Samko2003}.

 Similar to the classical Nirenberg's prescribing curvature problem, one may ask: {\it For which $Q_\alpha(x)$, does equation \eqref{3-6} have a positive solution?}

\medskip

%
%
%


\smallskip

\noindent{\bf 1.4. Existence for antipodally symmetric curvature function for $\alpha>n$}

\smallskip

There are two types of existence results concerning the prescribing curvature problem on $\mathbb{S}^n$ for $\alpha=2$ and $n \ge 2$ (the classic Nirenberg problem). For antipodally symmetric scalar curvature, the existences were obtained by Koutroufiotis \cite{Ko1972}, Moser \cite{Moser1973} (for $n=2$), by Escobar and Schoen (for $n=3$, and for $n\ge 4$ with curvature function also satisfying certain flatness conditions), and later by Chen \cite{Chen2006}, Jiang \cite{Jiang2010} and by Dou and Zhu \cite{DZ1} (for $n=1$); Under certain non-degenerate assumption, necessary conditions for the existence were found by  Chang and Yang \cite{CY1987}  (for $n=2$), Bahri and Coron \cite{BahriC1991} (for $n=3$), Li \cite{Li1995} (for $n\ge 4$),  and  by Ai, Chou and Wei \cite{ACW2000} (for  $n=1$).

For $\alpha=n$ and $n\ge 3$, similar prescribed curvature problem was studied by Wei and Xu \cite{WX1998}-\cite{WX2009}, where the integral equation is used to derive {\it a priori} estimates.
Such a prescribing curvature equation for $\alpha \in (0, 2)$ and $\alpha<n$  has been studied in recent papers by Q. Jin, Y.Y. Li and J. Xiong \cite{JLX2011a}-\cite{JLX2011b}, where they use the inverse (pseudo-differential) operator of $\tilde I_\alpha$ on $\mathbb{S}^n$ (due to the early work of Pavlov and Samko \cite{PavlovS1984}), and the existences for symmetric curvature function as well as for non-degenerate function were obtained.

For $\alpha<n$, all existence results for symmetric functions rely on a modern folklore theorem: two point blow up will generate large energy (principle of the concentration compactness).

  Here, we shall focus on the study of the integral operators, and the case $\alpha>n$. It was noted in Dou and Zhu \cite{DZ2} that for $\alpha>n$, the principle of the concentration compactness does not hold for energy functional $J_{\alpha, R} (f)$ (defined in next section). Nevertheless, in this paper, we will establish the existence for any positive antipodally symmetric function on $\mathbb{S}^n$ for any  $\alpha>n$.


\begin{thm} Assume $\alpha>n$. For any positive, continuous and antipodally symmetric function  $R(\xi)$,  there is  a positive  antipodally symmetric solution $u(\xi)\in C^{[\alpha]-n}(\mathbb{S}^n)$ to equation (\ref{new1-1}). \label{thm1-1}
\end{thm}

For $\alpha=2$ and $n=1$, the corresponding differential equation under a weak condition on positivity ($R(\xi)$ could be negative at some points),
was initially studied by Chen \cite{Chen2006}, and late by Jiang \cite{Jiang2010}, and Dou and Zhu \cite{DZ1},
where the existence results are obtained via the study of the minimizing sequence for energy functional $E_1(u)$. In all these early papers, the proofs rely on the size of first eigenvalue of the Lapalacian operator on $\mathbb{S}^1$. Since there is no equivalent differential form for equation \eqref{new1-1} for $\alpha \ne 2k$ and $n=1$, and the complication of the nodal set for higher order differential operators on higher dimensional sphere (such as $P_2(u)$ on $\mathbb{S}^3$), we need a completely different approach. In fact, we will obtain the existence of solution by showing the existence of minimizer to functional $J_{\alpha, R}$ defined below (thus, we need to assume the positivity of the curvature function $Q_\alpha$), which is motivated by and relies on our early results on the reversed Hardy-Littlewood-Sobolev ineqaulity \cite{DZ2}. As a simple corollary, we obtain the existence for the prescribing antipodally symmetric $Q-$ curvature.

\begin{crl}
For any positive, continuous and antipodally symmetric function  $R(\xi)$,  there is  a positive  antipodally symmetric solution $u(\xi)\in C^1(\mathbb{S}^3)$ to equation (\ref{new1-1-2}).
\label{crl1-1}
\end{crl}

%
%
%

It is interesting to point out that no extra condition on $R(\xi)$ is needed in Theorem \ref{thm1-1} even for higher dimensional sphere (contrary to the result of Escobar and Schoen \cite{ES1986}). On the other hand,  except for the case of $\alpha=2$ and $n=1$, the problem on the existence for the sign-changing curvature function is still open.

At the end of this paper, we also introduce a curvature function on a general compact Riemannian manifold,
and prove the conformal covariance property for the corresponding integral operator. A general Yamabe problem
(with the classical Yamabe problem as its special case) is also proposed.


\section{Variational approach}


\subsection{Existence via minimizing energy} We first give the proof of Theorem \ref{thm1-1}.

From now on, we always assume that $\alpha>n$, and $R(\xi)$ is a positive and continuous function on $\mathbb{S}^n$.

Write $f(x)=u^{\frac{n+\alpha}{n-\alpha}}(x)$, equation \eqref{new1-1} can be written as
$$
f^{\frac{n-\alpha}{n+\alpha}}(\xi)=\int_{\mathbb{S}^n}R(\xi)f(\eta)|\xi-\eta|^{\alpha-n}dS_\eta, \ \ \ f>0,
$$
which is equivalent to the following symmetric form:
\begin{equation}
R(\xi)f^{\frac{n-\alpha}{n+\alpha}}(\xi)=\int_{\mathbb{S}^n}R(\xi)R(\eta)f(\eta)|\xi-\eta|^{\alpha-n}dS_\eta, \ \ f>0. \label{new2-2}
\end{equation}
Define the quotient functional
$$
J_{\alpha, R}(f):=\frac {H_{\alpha, R}(f, f)}{||f||^2_{L_{\alpha, R}^{2n/(n+\alpha)}}},$$
where
 $$H_{\alpha, R}(f, f):=\int_{\mathbb{S}^n}\int_{\mathbb{S}^n}R(\xi)R(\eta)f(\xi)f(\eta)|\xi-\eta|^{\alpha-n}dS_\xi dS_\eta,$$
 and $$ ||f||_{L_{\alpha, R}^{2n/(n+\alpha)}}=\big (\int_{\mathbb{S}^n}f^{\frac{2n}{n+\alpha}}(\xi)R(\xi)dS_\xi \big)^{(n+\alpha)/2n}.$$
Easy to check that, up to a constant multiplier,
 a critical point to $J_{\alpha, R}(f)$ in $L^1(\mathbb{S}^n)$ is a weak solution to \eqref{new2-2}: For any continuous function $\phi(\xi)\in C(\mathbb{S}^n),$
\begin{equation}
\int_{\mathbb{S}^n} f^{\frac{n-\alpha}{n+\alpha}}(\xi)R(\xi) \phi(\xi)dS_\xi= \int_{\mathbb{S}^n}\int_{\mathbb{S}^n}R(\xi)R(\eta)f(\eta)|\xi-\eta|^{\alpha-n}\phi(\xi)dS_\eta dS_\xi.
\label{weak}
\end{equation}


To prove Theorem \ref{thm1-1}, we only need to establish the following

\begin{prop}If $\alpha>n$ and $R(\xi)$ is a positive, continuous and antipodally symmetric (i.e. $R(\xi)=R(-\xi))$ function, then
$$\inf_{f\in L^{1}(\mathbb{S}^n), f>0, f(\xi)=f(-\xi)}J_{\alpha, R}(f)
$$
is attained by a positive, antipodally symmetric function $f_\circ \in C^{[\alpha]-n}(\mathbb{S}^n).$
\label{infimum}
\end{prop}

We first show that the infimum is a positive value.

\smallskip

\begin{lm} Let $\alpha>n$. If $R(\xi)\in C(\mathbb{S}^n)$, then there is a positive constant $C_1$, depending on the minimum value of $R(\xi)$, such that $\inf_{f\in L^{\frac{2n}{n+\alpha}}(\mathbb{S}^n), f>0}J_{\alpha, R}(f) \ge C_1$.
\label{lm2-1}
\end{lm}
\begin{proof}
 Recall that  the following reversed Hardy-Littlewood-Sobolev inequality holds for any positive $f\in L^{\frac{2n}{n+\alpha}}(\mathbb{S}^n)$ (see \cite{DZ2}):
$$
\int_{\mathbb{S}^n}\int_{\mathbb{S}^n}f(\xi)f(\eta)|x-y|^{\alpha-n}dS_\xi dS_\eta \ge C_2 ||f||^2_{L^{\frac{2n}{n+\alpha}}(\mathbb{S}^n)}
$$
for some universal constant $C_2>0$.
Thus
$$
J_{\alpha, R}(f) \ge [\min_{\xi\in \mathbb{S}^n} R(\xi)]^2 \cdot C_2:=C_1.
$$

\end{proof}

We are now ready to prove Proposition \ref{infimum} (the infimum is attained). The proof is along the line of our proof for the sharp reversed Hardy-Littlewood-Sobolev inequality in \cite{DZ2}.

Let $\{f_j\}_{j=1}^\infty \in L^1$  be a positive minimizing sequence with $\|f_j\|_{L^{2n/(n+\alpha)}}=1.$
We also need the following density lemma.

\begin{lm}
{\bf (Density Lemma)} Let $F(\theta) \in  L^1( \mathbb{S}^n)$ be a nonnegative, antipodally symmetric function with $\|F\|_{L^{2n/(n+\alpha)}( \mathbb{S}^n)}=1$. For any $\epsilon>0$, there is a nonnegative, antipodally symmetric function $G(\xi)\in C^0( \mathbb{S}^n)$, such that
$$
||F-G||_{L^{2n/(n+\alpha)}( \mathbb{S}^n)}+\big{|}H_{\alpha, R}(F, F)- H_{\alpha, R}(G, G)\big {|}<\epsilon.
$$
\label{lm2-2}
\end{lm}
\begin{proof}  Let $\{G_i\}_{i=1}^\infty$ be a sequence of nonnegative, antipodally symmetric and continuous functions such that $||G_i-F||_{L^1( \mathbb{S}^n)} \to 0$ as $i\to \infty$. Then, for any $\xi\in \mathbb{S}^n,$ as $i\to \infty$,
\begin{eqnarray*}
\big{|}H_{\alpha, R}(F, F)- H_{\alpha, R}(G, G)\big {|}   & \le& C \int_{\mathbb{S}^n} \int_{\mathbb{S}^n} |G_i(\eta)-F(\eta)|\cdot |G_i(\xi)-F(\xi)|d\eta d \xi \\
&\le& C \big \{\int_{\mathbb{S}^n} |G_i(\eta)-F(\eta)| d\eta \big \}^2\to 0.
\end{eqnarray*}
Lemma \ref{lm2-2} immediately follows from the above.
\end{proof}

Due to Lemma \ref{lm2-2}, we can further assume $f_j  \in C (\mathbb{S}^n)$.

Note that $\|f_j\|_{L^{2n/(n+\alpha)}}=1$ and $f_j$ is antipodally symmetric, up to a subsequence, we know, without loss of generality, that there is a small positive number $\delta_0$ such that
\begin{equation}
\int_{\mathbb{S}^n \cap B_{\delta_0}(\mathbb{N})} f_j dy\ge 1/100, \ \ \ \int_{\mathbb{S}^n \cap B_{\delta_0}(\mathbb{S})} f_j dy\ge 1/100,
\label{4-10-1}
\end{equation}
where $B_{\delta_0}(\mathbb{N}), \ B_{\delta_0}(\mathbb{S})$ represent the geodesic ball centered at the north pole and at the south pole, respectively.

Denote
\begin{equation}
I_{\alpha, R}f(\xi)=\int_{\mathbb{S}^n}R(\eta)f(\eta)|\xi-\eta|^{\alpha-n}dS_\eta.
\label{I-int}
\end{equation}
We then know, due to \eqref{4-10-1} that there is a universal positive constant $C_4>0$, such that
\begin{equation}
I_{\alpha, R}f_i(\xi) \ge C_4 \ \ \ \ \mbox{for \ \ \ all} \ \ \ \xi \in \mathbb{S}^n.
\label{4-10-2}
\end{equation}
On the other hand, if $meas\{\xi\in \mathbb{S}^n \ : I_{\alpha, R}f_i(\xi) \to \infty \ \mbox{as} \  i\to \infty\}=vol(\mathbb{S}^n),$ then we have, using \eqref{4-10-1}, that $H_{\alpha, R}(f_i, f_i) \to \infty$, which contradicts the assumption that $f_i$ is a minimizing sequence. Thus $I_{\alpha, R}f_i(\xi)$ stays uniformly bounded in a set with positive measure. This implies: there is a constant $C_5>0$, such that
\begin{equation}
\int_{\mathbb{S}^n}f_i(\xi) dS_\xi \le C_5.
\label{4-10-3}
\end{equation}
From \eqref{4-10-3} we know that sequence $\{I_{\alpha, R}f_i(\xi)\}_{i=1}^\infty$ is uniformly bounded and equiv-continuous on $\mathbb{S}^n$. Up to a subsequence, $I_{\alpha, R}f_i(x) \to L(x)\in C(\mathbb{S}^n)$.
Using Fatou Lemma and the reversed Hardy-Littlewood-Sobolev inequality (see Dou and Zhu \cite{DZ2}), we have, up to a further subsequence, for $m\in \mathbb{N}$, that
\begin{eqnarray*}
0&\ge&\big(\lim_{i\to \infty}\int_{\mathbb{S}^n}| I_{\alpha, R}f_i-I_{\alpha, R}f_{i+m}|^{2n/(n-\alpha)}  \big)^{(n-\alpha)/2n}\\
&\ge& C (\lim_{i\to \infty}||f_i-f_{i+m}||^{2n/(n-\alpha)}_{L^{2n/(n+\alpha)}})^{(n-\alpha)/2n}.
\end{eqnarray*}
Thus $||f_i-f_{i+m}||_{L^{2n/n+\alpha)}}\to 0$. This implies $||f_i-f_\circ||_{L^{2n/(n+\alpha)}}\to 0$ for some $f_\circ$ with $||f_\circ||_{L^{2n/(n+\alpha)}}=1.$ Thus,
 up to a further subsequence, $f_i \to f_\circ \ge 0$ almost everywhere, and $f_\circ(\xi)=f_\circ(-\xi).$
It follows, via Fatou Lemma, that $\lim_{j\to \infty}H_{\alpha, R}(f, f)
\ge H_{\alpha, R}(f_\circ, f_\circ).$
Thus the infimum  is achieved by $f_\circ \ge 0$. From \eqref{4-10-3}, we know that $f_\circ \in L^1(\mathbb{S}^n)$.  Thus $I_{\alpha, R}f_\circ(\xi) \le C<\infty.$ It follows that $f_\circ>0$, and up to a constant multiplier, $f_\circ$ satisfies \eqref{weak}. This yields:
$$f_\circ^{\frac{n-\alpha}{n+\alpha}}(\xi)=\int_{\mathbb{S}^n}R(\xi)f_\circ(\eta)|\xi-\eta|^{\alpha-n}dS_\eta \, \, \mbox{ a.e. \  on } \ \mathbb{S}^n.
$$
 $f_\circ$ is in $C^{[\alpha]-n} (\mathbb{S}^n).$

The proof of Proposition \ref{infimum} is completed.

%

\subsection{Conformal curvature and covariance}

For $n=1$ and $\alpha =2$, the analog ``scalar" curvature under a new metric $g=v^{-4}g_0$ on standard sphere $(\mathbb{S}^1, g_0=d \theta \otimes  d\theta)$ was defined in Ni and Zhu \cite{NZ1}:
$$
R_g=v^3(4 v_{\theta\theta}+v).
$$
Thus $R_{g_0}=1.$
The corresponding {conformal Lapalace-Beltrami}
operator of $g$ is then defined by
$$
L_g=4\Delta_g+R_g.
$$
Then the conformal covariance property for $L_g$  was proved \cite{NZ1}: For $\varphi>0$, if $g_2=\varphi^{-4}g_1$ then
$R_{g_2}=\varphi^3 L_{g_1}\varphi$, and
\begin{equation}
L_{g_2}(\psi)=\varphi^3 L_{g_1}(\psi\varphi),
\quad\forall \psi\in C^2({\mathbb{S}^1}).
\label{4-23-1}
\end{equation}
A similar conformal operator involving fourth order derivative (corresponding to $\alpha=4$ in this paper) was also introduced in \cite{NZ1}.

Notice that inverse operator of $L_{g_0}$ on $\mathbb{S}^1$ in fact is $\tilde I_2$ (see definition \eqref{3-6}). If we write
$$\tilde I_{\mathbb{R}^n, g_0, \alpha} (u)= c_{n,\alpha}\int_{\mathbb{R}^{n}}{|x-y|^{\alpha-n} {u}(y)}dy
$$
and, for $g_1=(\frac 2{1+|x|^2})^2 g_0=\phi^{4/(n-\alpha)}g_0$,
$$\tilde I_{\mathbb{S}^n, g_1, \alpha} (u)= c_{n,\alpha}\int_{\mathbb{S}^{n}}{|\xi-\eta|^{\alpha-n} {u}(\eta)}dV_{g_1},
$$
then covariance relation \eqref{4-23-1} implies, for $\alpha=2$ and $n=1$, that,
$$
\tilde I_{\mathbb{S}^1, g_1, 2} (u)= \phi^{-1} \tilde I_{\mathbb{S}^1, g_0, 2} (\phi^{-1} \cdot u)
\quad\forall u \in C^2({\mathbb{S}^1}).
$$

It is quite difficult, if it is possible,  to derive any covariance relation similar to \eqref{4-23-1} for other parameter $\alpha$ if it is not an even number  since we do not have the precise differential form. However, we can derive the  covariance relation for the integral operator $\tilde I_{\mathbb{S}^n, g_1, \alpha} (u)$ for any positive $\alpha$:

\begin{prop} Let $g_e=\sum_{i=1}^n d x_i \otimes d x_i$ be the standard flat metric on $\mathbb{R}^n$. Let $\mathcal{S}:\, x \in \mathbb{R}^n\to \xi\in \mathbb{S}^n\backslash(0,0,\cdots,-1)$ be the inverse of a stereographic projection, defined by
\begin{eqnarray*}
\xi^j:=\frac{2x^j}{1+|x|^2},\quad\text{for}~j=1,2,\cdots,n;\quad\xi^{n+1}:=\frac{1-|x|^2}{1+|x|^2}.
\end{eqnarray*}
For any positive $\alpha \ne n$, if $g_1=\phi^{\frac 4{n-\alpha}}g_e$, where  $\phi(x)=\big(\frac 2{1+|x|^2} \big )^{\frac {n-\alpha}{2}}$, then
\begin{equation}
\tilde I_{\mathbb{S}^n, g_1, \alpha} (u (\mathcal{S}))= \phi^{-1} \tilde I_{\mathbb{R}^n, g_0, \alpha} (\phi^{\frac{n+\alpha}{n-\alpha}} \cdot u)
\quad\forall u \in C^0({\mathbb{R}^n}).
\label{covariance}
\end{equation}
\label{prop4-1}
\end{prop}

\begin{proof}

Easy to check (or, see e.g.  \cite {Lieb1983, LL2001}): for  $x,y\in \mathbb{R}^n,\xi\in \mathbb{S}^n,$
\begin{eqnarray}\label{lieb1}
|\mathcal{S}(x)-\mathcal{S}(y)|=\big[\frac{4|x-y|^2}{(1+|x|^2)(1+|y|^2)}\big]^\frac12,
\end{eqnarray}
So,
\begin{eqnarray*}
\tilde I_{\mathbb{S}^n, g_1, \alpha} (u)&=& \int_{\mathbb{S}^n}|\mathcal{S}(x)-\mathcal{S}(y)|^{\alpha-n} u(\mathcal{S}(y))d V_{g_1}\\
&=&\big(\frac 2{1+|x|^2} \big )^{\frac {\alpha-n}{2}}\int_{\mathbb{R}^n}|x-y|^{\alpha-n} u(y)\cdot \big(\frac 2{1+|y|^2} \big )^{\frac {\alpha+n}{2}}d V_{g_0}\\
& =& \phi^{-1} \tilde I_{\mathbb{R}^n, g_0, \alpha} (\phi^{\frac{n+\alpha}{n-\alpha}} \cdot u).
\end{eqnarray*}
\end{proof}

Observe, in identity \eqref{lieb1}, that $G^{g_0} (x,y):=c_{n,2}|x-y|^{n-2}$ is the Green's function for conformal Lapalacian operator $-\Delta_x$ and $G^{g_1}(\mathcal{S}(x), \mathcal{S}(y)):=c_{n,2}|\mathcal{S}(x)-\mathcal{S}(y)|^{n-2}$ is the Green's function for conformal Lapalacian operator $L_{g_1}=-\Delta_{g_1} +c(n) R_{g_1}$  for $g_1= \phi^{\frac 4{n-\alpha}} g_0,$ where and throughout the rest of this section $c(n)=\frac{n-2}{4(n-1)}$ and $R_{g_1}$ is the scalar curvature under metric $g_1$. It turns out, this type of identity holds for general manifolds with positive scalar curvature.

\begin{prop} For a given compact Riemannian manifold $(M^n, g_0)$ (for $n\ne 2)$, let $R_{g_0}$ be its scalar curvature. Let $g_1=\phi^{\frac {4}{n-2}}g_0$ be a conformal metric with $R_{g_1}>0$ being its scalar curvature. If the Green's function for the conformal Laplacian operator $L_{g_0}=-\Delta_{g_0}+c(n)R_{g_0}$ is $G^{g_0}(y, x)$, then the Green's function for the conformal Laplacian operator $L_{g_1}=-\Delta_{g_1}+c(n)R_{g_1}$ is given by
\begin{equation}
G^{g_1}(y, x)=\phi^{-1}(y)\phi^{-1}(x) G^{g_0}(y, x).
\label{Green}
\end{equation}
\label{green}
\end{prop}
\begin{proof}For any  function $u \in C^2(M^n)$, denote
$$I_{g_0}(u)=\int_{M^n}G^{g_0}(y, x)u(y) dV_{g_0}(y), \ \ \ \ I_{g_1}(u)=\int_{M^n}G^{g_1}(y, x)u(y) dV_{g_1}(y).
$$
Using the conformal covariance for the conformal Laplacian operator, we have,
\begin{eqnarray*}
L_{g_1}\big ( \phi^{-1}I_{g_0}(\phi^{\frac{n+2}{n-2}} u) \big)&=& \phi^{-\frac{n+2}{n-2}}L_{g_0}\big ( I_{g_0}(\phi ^{\frac{n+2}{n-2}} u) \big)\\
&=&u(x).
\end{eqnarray*}
Thus
\begin{eqnarray*}
\int_{M^n}G^{g_1}(y, x))u(y) dV_{g_1}(y)&=& \phi^{-1}(x)I_{g_0}(\phi^{\frac{n+2}{n-2}} u)\\
&=&\int_{M^n}\phi^{-1}(x)G^{g_0}(y, x))u(y) \phi^{\frac{n+2}{n-2}}(y) dV_{g_0}(y)\\
&=&\int_{M^n}\phi^{-1}(x)G^{g_0}(y, x))\phi^{-1}(y)u(y) dV_{g_1}(y).
\end{eqnarray*}
The proposition follows from the above.
\end{proof}

 Based on this proposition, for any compact Riemannian manifold $(M^n, g_0)$ ($n\ne 2$) with positive scalar curvature $R_{g_0}$ and parameter $\alpha \ne n+2k$ for $k=0, 1, \cdots$, we may introduce $\alpha-$ curvature $Q_{\alpha, g}$ under the conformal metric $g=\phi^{\frac 4{n-\alpha}}g_0$ as follows.

 Let $G^{g_0}(y, x)$ be the Green's function for the (second order) conformal Laplacian operator $-\Delta_{g_0} +c(n)R_{g_0}$. Define
  \begin{equation}\label{new_operator}
   I_{M^n, g_0, \alpha} (f)(x)=\int_{\mathbb{M}^{n}}[G^{g_0}(y, x)]^{\frac {\alpha-n}{2-n}} f(y) dV_{g_0}.
  \end{equation}
Similar operators and related proposal were discussed in \cite{DZ1-1}, however this part was cut in the journal version \cite{DZ1paper} on the request of the referee.
 It follows from Proposition \ref{green} that  $I_{M^n, g_0, \alpha}$ has the following conformal covariance property:
\begin{thm}
For any positive $\alpha \ne n$, if $g_1=\phi^{\frac 4{n-\alpha}}g_0$, then
\begin{equation}
I_{M^n, g_1, \alpha} (u)= \phi^{\frac{\alpha-n}{n-2}} I_{M^n, g_0, \alpha} (\phi^{\frac{2n}{n-\alpha}+\frac{\alpha-n}{n-2}} \cdot u),
\quad\forall u \in C^0({\mathbb{M}^n}).
\label{covariance-1}
\end{equation}
\label{covariance_thm}
\end{thm}
\begin{proof}
\begin{eqnarray*}
I_{M^n, g_1, \alpha} (u)&=&\int_{\mathbb{M}^{n}}[G^{g_1}(y, x)]^{\frac {\alpha-n}{2-n}} u(y) dV_{g_1}\\
&=&\phi^{\frac{\alpha-n}{n-2}} \int_{\mathbb{M}^{n}}[G^{g_0}(y, x)]^{\frac {\alpha-n}{2-n}}\phi^{\frac{2n}{n-\alpha}+\frac{\alpha-n}{n-2}} u(y) d V_{g_0}.
\end{eqnarray*}
\end{proof}

For a given compact Riemannian manifold $(M^n, g_0)$ ($n \ne 2)$ with positive scalar curvature, the $\alpha-$ curvature $Q_{\alpha,{g_1}}$ (for $\alpha \ne n)$ under the conformal curvature $g_1=\phi^{\frac 4{n-\alpha}}g_0$ is then defined as the  function which satisfies
 \begin{equation}\label{3-7}
\phi(x)= \int_{\mathbb{M}^{n}}[G^{g_0}(y, x)]^{\frac {\alpha-n}{2-n}} {Q_{\alpha,{g_1}} (y) \phi^{\frac{n+\alpha}{n-\alpha}}(y)}dV_{g_0}= I_{M^n, g_0, \alpha} ({Q_{\alpha,{g_1}} \phi^{\frac{n+\alpha}{n-\alpha}}}).
\end{equation}
It is clear that $Q_{2, g_1}$ is in fact the scalar curvature. Is $Q_{4. g_1}$ the $Q-$ curvature (may up to a universal constant)? This is a topic we will study in the future. On the other hand,  if  $I_{M^n, g_0, \alpha}$ is invertible, then its inverse operator may yield a conformal differential operator with higher order ($\alpha$-th order). It is interesting to seek the relation between our integral operator defined in \eqref{new_operator} with  GJMS operator\footnote{While the current paper is circulated, I was called attention to Hang and Yang's recent paper \cite{HY2014-2}, and realized that this question is answered by Proposition 1.1 in their paper.}. For a given compact Riemannian manifold $(M^n, g_0)$ ($n \ne 2)$ with positive scalar curvature, the corresponding Yamabe type problem for $\alpha \ne n$  can be formulated as: whether there is a conformal metric $g=\phi^{\frac 4{n-\alpha}}g_0$ so that the $Q_{\alpha,{g}}$ under this metric is constant. This problem will be addressed when we extend the classic sharp Hardy-Littlewood-Sobolev inequality of Lieb on general compact Riemannian manifolds \cite{HZ2014}.

\medskip

\noindent{\bf ACKNOWLEDGMENT.} It's my great pleasure to thank Professor P. Yang who brought my attention to the project of understanding curvature equations with negative power ten more years ago while we were working on
 the sharp Sobolev inequality for Paneitz operator on $\mathbb{S}^3$, and thank him  for sending me their recent preprints \cite{HY2014-1}-\cite{HY2014-2}. Only after we obtained
 the sharp reversed Hardy-Littlewood-Sobolev inequality a couple of years ago in \cite{DZ2}, we realized
 the deepness, richness of the project. Thanks also go to Professor  A. Chang for the conversation on early existence results on the classic prescribing curvature equation.  This work  is partially supported by a collaboration grant from Simons Foundation.


\end{document}